\documentclass{amsart}

\usepackage{amssymb,color}
\usepackage{amsfonts}
\usepackage{amsmath}
\usepackage{euscript}
\usepackage{enumerate}
\usepackage{pdfsync}
\newtheorem{theorem}{Theorem}[section]
\newtheorem{lemma}[theorem]{Lemma}

\newtheorem{cor}[theorem]{Corollary}

\newtheorem*{Theorem1'}{Theorem 1'}

\theoremstyle{definition}

\theoremstyle{remark}

\newcommand \GL{{\mathrm{GL}}}
\newcommand \Z{{\mathbb Z}}
\newcommand \Q{{\mathbb Q}}
\newcommand \N{{\mathbb N}}

\begin{document}

\title[Linear systems of diophantine equations]{Linear systems of diophantine equations}

\author{F. Szechtman}
\address{Department of Mathematics and Statistics, University of Regina, Canada}
\email{fernando.szechtman@gmail.com}
\thanks{The author was supported in part by an NSERC discovery grant}

\subjclass[2020]{11D04, 15A06}



\keywords{Linear system; Diophantine equation; Smith normal form}

\begin{abstract} Given free modules $M\subseteq L$ of finite rank $f\geq 1$ over a principal ideal domain $R$, we
give a procedure to construct a basis of $L$ from a basis of $M$ assuming the invariant factors
or elementary divisors of $L/M$ are known. Given a matrix $A\in M_{m,n}(R)$ of rank $r$, its nullspace~$L$ in $R^n$
is a free $R$-module of rank~$f=n-r$. We construct a free submodule $M$ of $L$ of rank~$f$ naturally associated to $A$
and whose basis is easily computable, we determine the
invariant factors of the quotient module $L/M$, and then indicate how to apply the previous procedure to build a basis of $L$ from one of $M$.
\end{abstract}

\maketitle

\section{Introduction}

Let $R$ be a principal ideal domain. Given $f\in\N$, by a lattice of rank $f$ we understand a free $R$-module $L$ of rank $f$. By a sublattice of $L$ we mean
a submodule $M$ of $L$, necessarily free, also of rank~$f$. In this case, $L/M$ is a finitely generated torsion $R$-module.

In different settings, we may face the problem of having to construct
a basis of $L$ from a known basis $\{u_1,\dots,u_f\}$ of $M$. A prime example occurs when $R=\Z$, $L={\mathcal O}_K$, the ring
of integers of an algebraic number field $K$ of degree $f$ over $\Q$, $M=\Z[\theta]$, and $\{u_1,\dots,u_f\}=\{1,\theta,\dots,\theta^{f-1}\}$,
where $\theta\in {\mathcal O}_K$ is chosen so that $\Q[\theta]=K$.

A general procedure to construct
a basis of $L$ from a known basis $\{u_1,\dots,u_f\}$ of $M$ is available to us, provided we know the index $i(M,L)$ of $M$ in $L$,
which is the determinant of the matrix whose columns are the coordinates of
any basis of $M$ relative to any basis of $L$. This is determined up to multiplication by units only. Note that if
$R=\Z$, then $|i(M,L)|$ is the order of the finite abelian group $L/M$.

If $i(M,L)=1$ then $L=M$ has basis $\{u_1,\dots,u_f\}$. Suppose next $i(M,L)\neq 1$ and let $p\in R$ be a prime factor of $i(M,L)$. Then $L/M$ has a cyclic submodule isomorphic to $R/Rp$, so there exist $a_1,\dots,a_f\in \Z$, such that
\begin{equation}
\label{p}
v=\frac{a_1u_1+\cdots+a_f u_f}{p}\in L\setminus M.
\end{equation}
Since $v\notin M$, we have $p\nmid a_i$ for some $i$ and we assume for notational convenience that $i=1$.
Since $\gcd(p,a_1)=1$, we can find $x,y\in R$ such that $xa_1+yp=1$. Here $ypv\in M$, so $$xa_1v=(1-yp)v=v-ypv\notin M,$$
whence $xv\notin M$ and a fortiori $v'=xv+yu_1\in L\setminus M$, where
$$
v'=\frac{x(a_1u_1+\cdots+a_f u_f)+pyu_1}{p}=\frac{u_1+xa_2u_2+\cdots+xa_f u_f}{p}.
$$
Thus, replacing $v$ by a suitable $R$-linear combination of itself and $u_1$, namely $v'$, we may suppose that $a_1=1$ in (\ref{p}).
Then $\{v,u_2,\dots,u_f\}$ is a basis for a sublattice, say $P$, of $L$ such that
$M\subset P$ and $i(M,P)=p$, so $i(P,L)=i(M,L)/p$. Repeating this process we eventually arrive at a basis of~$L$.

In this paper we modify and improve this procedure, provided the invariant factors or elementary divisors of $L/M$ are known,
and we illustrate the use of this method with a concrete problem.

Indeed, let $F$ be the field of fractions of $R$.
Given a matrix $A\in M_{m,n}(R)$, we write $T$
for the nullspace of $A$ in $F^n$, so that $S=T\cap R^n$ is the nullspace of $A$ in $R^n$.
We note that $L=S$ is a lattice of rank $f=n-r$, where $r$ is the rank of $A$.

It is easy to find an $F$-basis of $T$ from the reduced row echelon form, say $E$, of $A$.
It is not clear at all how to use $E$ to produce an $R$-basis of $S$. To achieve this, we first
identify a sublattice $M$ of $L$ as well as a basis of $M$, naturally, in terms of $E$; we then compute the complete structure of the $R$-module $L/M$, namely its invariant factors, whose product is equal to $i(M,L)$; we finally indicate
how to build a basis of $L$ from the given basis of $M$ by making use of the full structure of $L/M$.

Now if $r=n$ then $E$ consists of the first $n$ canonical vectors of $R^m$, and $L=0$. On the other hand, if $r=0$ then $E=0$ and $L=R^n$.
None of these cases is of any interest, so we assume throughout that $0<r<n$.

In \S\ref{ini}, \S\ref{wpq}, \S\ref{eachother}, and~\S\ref{main}, we use $E$ to naturally produce a lattice $U$ of rank $f$, a basis of $U$, and a nonzero scalar $d\in R$ such that $M=dU$ is a sublattice of $L$, and $L$ is a sublattice of $U$. Moreover, we compute
the full structures of $L/M$ and $U/L$. Furthermore, in \S\ref{algo} we indicate how to use either the invariant factors or the elementary divisors of $L/M$ to construct a basis of $L$ from one of $M$ (this is done for arbitrary $L$ and $M$).
In addition, if $d=p$ is a prime, we indicate in~\S\ref{dp} how to produce a basis of $L$ more or less directly from one of~$U$. Examples can be found in \S\ref{exa}.

We may summarize our study of $S$ as follows: given the lattice $S$ of all solutions of $AX=0$ in $R^n$, we approximate $S$ from below by a
naturally occurring lattice of solutions $M$ in $R^n$, we determine how far
$M$ is from $S$, and we describe how to bridge the gap between them. A like approach was recently utilized in \cite{QSS} in the special case
of a single linear homogeneous equation, that is, when $m=1$, except that in \cite{QSS} the approximation was taken from above, by means of~$U$.
The case $m=1$ is necessarily simpler than the general case addressed here, as much in the computation of the structures of $U/S$ and $S/M$
as in the passage from a basis of a lattice to a basis of $S$, where the material from \S\ref{algo} not required.

As is well-known (see the note at the end of \cite[\S 4]{QSS} in the special case $m=1$), we may also find a basis of $S$ by appealing to the Smith normal form $D$ of $A$. There are $P\in\GL_m(R)$ and $Q\in\GL_n(R)$ such that $D=PAQ$. It is then trivial
to find a basis, say ${\mathcal B}$, of the nullspace of~$D$, whence $Q{\mathcal B}$ is a basis of the nullspace of $A$. This approach gives no information
whatsoever on how far naturally occurring lattices of solutions of $AX=0$ are from $S$, as provided in Theorem \ref{quot}, or how to expand or shrink these lattices to reach $S$, as expounded in \S\ref{algo} or \cite[Theorem 4.5]{QSS}.


Most of the literature on systems of linear diophantine equations is naturally focused on the case $R=Z$. One significant body of work
is focused on non-negative solutions, with applications to linear programming and combinatorial optimization. See \cite{CD},
\cite{CF},  \cite{CKO}, \cite{GK}, \cite{PV}, \cite{S}, and references therein. 

Regarding lattices over the integers and their bases, a large body of literature is concerned with lattice basis reduction, which
takes as input a basis of a lattice and aims at producing as output a new basis of the same lattice with vectors that are
short and nearly orthogonal. A celebrated algorithm of this kind is the LLL algorithm \cite{LLL}, which has a wide range of applications, such
as in cryptanalysis, algorithmic number theory, factorization of polynomials with rational coefficients, integer linear programming, and many more. See the reference book \cite{PV2} for comprehensive information on this subject.


\section{reduced matrices}\label{ini}

A matrix $Z\in M_{m,n}(R)$ of rank $r$ is said to be {\em reduced} if there are $0\neq d\in R$ and $K\in M_{r,f}(R)$ such that
\begin{equation}
\label{ready}
Z=\left(\begin{array}{cc} d I_r & K\\ 0 & 0
\end{array}\right).
\end{equation}
Two matrices $B,C\in M_{m,n}$ are {\em associated} if there there are $L\in\GL_m(F)$ and $\Sigma\in\GL_n(F)$
such that $\Sigma$ is a permutation matrix and $LB\Sigma=C$. This is clearly an equivalence relation.

\begin{lemma}\label{lista} The given matrix $A$ is associated to a reduced matrix.
\end{lemma}

\begin{proof} Let $Y\in M_{m,n}(F)$ be the reduced row echelon form of $A$. Multiplying $Y$ by suitable element of $R$
and permuting the columns of resulting matrix yields a reduced matrix associated to $A$.
\end{proof}

For the remainder of the paper we fix a reduced matrix $Z$  associated to $A$, say via
that $LA\Sigma=Z$, and write $J=(d I_r\, K)\in M_{r,n}(R)$ for the matrix obtained from $Z$ by eliminating its last $m-r$ rows.
We let $N$ stand for the nullspace of $J$ in $R^n$, so that $S=\Sigma N$ (thus, up to permutation of the variables $X_1,\dots,X_n$,
our linear system is $JX=0$).

\section{Choice of a lattice}\label{wpq}

The linear system $JX=0$ reads as follows:
$$
d X_{1}=-(K_{1,1}X_{r+1}+\cdots+K_{1,f}X_{n}),
$$
$$
dX_{2}=-(K_{2,1}X_{r+1}+\cdots+K_{2,f}X_{n}),
$$
$$
\vdots
$$
$$
dX_{r}=-(K_{r,1}X_{r+1}+\cdots+K_{r,f}X_{n}).
$$
Consider the $f$ vectors $V(1),\dots,V(f)\in F^n$ and defined as follows:
\begin{equation}
\label{vi}
V(1)=\left(\begin{array}{c} -\frac{K_{1,1}}{d}  \\ \vdots \\ -\frac{K_{r,1}}{d} \\ 1 \\ 0 \\ \vdots \\ 0\end{array}\right),
\dots,
V(f)=\left(\begin{array}{c} -\frac{K_{1,f}}{d}  \\ \vdots \\ -\frac{K_{r,f}}{d} \\  0 \\  \vdots \\ 0 \\  1\end{array}\right).
\end{equation}
It is clear that $\{V(1),\dots,V(f)\}$ is an $F$-basis of the nullspace of $J$ in $F^n$. We set
$$
W=\mathrm{span}_R\{V(1),\dots,V(f)\},
$$
so that $\{V(1),\dots,V(f)\}$ is an $R$-basis of $W$. We thus have
\begin{equation}
\label{faltaba}
d W\subseteq N\subseteq W,
\end{equation}
and we aim to determine the structure of the factors
$$
N/dW\text{ and }W/N,
$$
where
$$
W/dW\cong (R/Rd)^{f}.
$$

Given $\alpha_1,\dots,\alpha_{f}\in F$, we have
\begin{equation}\label{cong}
\alpha_1 V(1)+\cdots+\alpha_{f}V(f)\in N\Leftrightarrow\alpha_1,\dots,\alpha_{f}\in R\text{ and }
\alpha_1 K_{i,1}+\cdots+\alpha_{f} K_{i,f}\equiv 0\mod d,\;1\leq i\leq r.
\end{equation}
Thus, we have an isomorphism $R^{f}\to W$ given by
$$(\alpha_1,\dots,\alpha_{f})\mapsto \alpha_1 V(1)+\cdots+\alpha_{f}V(f),$$
and $N$ corresponds to the submodule, say $Y$, of $R^{f}$ of all $(\alpha_1,\dots,\alpha_{f})$ such that
the right hand side of (\ref{cong}) holds. In particular, $W/N \cong R^f/Y$.

\section{Each of $N/dW$ and $W/N$ determines the other}
\label{eachother}

By the theory of finitely generated modules over a principal ideal domain, there is a basis $\{u_1,\dots,u_{f}\}$ of $W$ and non-zero elements
$a_1,\dots,a_f\in R$ such that
$$
a_1|\cdots|a_f|d
$$
and
$\{a_1u_1,\dots,a_f u_{f}\}$ is a basis of $N$. Since $\{du_1,\dots,d u_{f}\}$ is a basis of $dW$, we see that
$$
W/N\cong (R/Ra_1)\oplus\cdots\oplus (R/Ra_f)\text{ and }N/dW\cong (R/Rb_f)\oplus\cdots\oplus (R/Rb_1),
$$
where
$$
b_i=d/a_i,\quad 1\leq i\leq f,
$$
and
$$
b_f|\cdots|b_1.
$$
As $d$ is fixed, we see that $N/dW$ and $W/N$ determine each other.

\section{Structures of $W/N$ and $N/dW$}
\label{main}

Set $\overline{R}=R/Rd$ and consider the homomorphism of $R$-modules
$$
\Delta:R^f\to R^r\to \overline{R}^r
$$
given by
$$
\alpha\mapsto \overline{K}\overline{\alpha},
$$
where $\alpha=(\alpha_1,\dots,\alpha_f)$, and $\overline{K}$ and $\overline{\alpha}$ are the reductions of $K$ and $\alpha$ modulo $Rd$.
Then (\ref{cong}) shows that the kernel of $\Delta$ is $Y$. Thus
$$
W/N\cong R^f/Y\cong \Delta(R^f)\cong C(\overline{K}),
$$
where $C(\overline{K})$ is the column space of $\overline{K}$, namely the $\overline{R}$-span of the columns of $\overline{K}$.

Consider the natural epimorphism of $R$-modules $\Lambda:R^r\to\overline{R}^r$ with kernel $(Rd)^r$. Then $\Lambda$
restricts to an epimorphism of $R$-modules $\Omega:C(K)\to C(\overline{K})$ with kernel $C(K)\cap (Rd)^r$.

Let $Q=\mathrm{diag}(q_1,\dots,q_s)$ be the Smith Normal Form of $K$, where $q_1|\cdots|q_s$ and $s=\min\{r,n-r\}$,
and let $t$ be the rank of $K$, so that $t=0$ if and only if $K=0$.

If $K=0$ then (\ref{cong}) implies that $W=N$ and a fortiori
$$N/d W\cong W/d W\cong \overline{R}^{f}.$$

Suppose next $K\neq 0$. Then $t$ is the last index such that $q_t\neq 0$ and from the theory of
finitely generated modules over a principal ideal domain, there is a basis $\{u_1,\dots,u_r\}$ of $R^r$ such that
$\{q_1 u_1,\dots,q_t u_t\}$ is a basis for $C(K)$.
Notice that
$$
C(K)\cap (Rd)^r=(Rq_1u_1\oplus\cdots\oplus Rq_tu_t)\cap (Rdu_1\oplus\cdots\oplus Rdu_r)=
R\,\mathrm{lcm}(d,q_1)u_1\oplus\cdots\oplus R\,\mathrm{lcm}(d,q_t)u_t,
$$
so that
$$
W/N\cong C(\overline{K})\cong C(K)/(C(K)\cap (Rd)^r)\cong (Rq_1u_1\oplus\cdots\oplus Rq_tu_t)/
(R\,\mathrm{lcm}(d,q_1)u_1\oplus\cdots\oplus R\,\mathrm{lcm}(d,q_t)u_t).
$$
Since
$$
\mathrm{lcm}(d,q_i)/q_i=d/\gcd(d,q_i),\quad 1\leq i\leq t,
$$
setting
$$
m_i=\mathrm{lcm}(d,q_i)/q_i,\; d_i=d/\gcd(d,q_i),\quad 1\leq i\leq t,
$$
we infer
\begin{equation}
\label{form}
W/N\cong R/R m_t\oplus\cdots\oplus R/R  m_1\cong
R/R d_t\oplus\cdots\oplus R/R d_1.
\end{equation}
Adding $f-t$ zero summands to the right hand side of (\ref{form}), we may write
$$
W/N\cong (R/R\cdot 1)^{f-t}\oplus R/R d_t\oplus\cdots\oplus R/R d_1.
$$
We finally deduce from \S\ref{eachother} the sought formula:
\begin{equation}
\label{formula}
N/Wd\cong R/R\gcd(d,q_1)\oplus\cdots\oplus R/R\gcd(d,q_t)\oplus (R/R d)^{f-t}.
\end{equation}

Dividing every entry of $Z$ by $g=\gcd\{d,K_{ij}\,|\, 1\leq i\leq r, 1\leq j\leq f\}$ we may assume that $g=1$,
which translates into $\gcd(d,q_1)=1$. In this case, if $r=1$ then (\ref{form}) and (\ref{formula}) reduce to the corresponding formulas from \cite[Theorems 4.1 and 3.2]{QSS}, respectively.

Notice that (\ref{form}) and (\ref{formula}) remain valid when $K=0$.

Set $U=\Sigma W$, with $\Sigma$ as in \S\ref{ini}, and let $M=d U$. We have an isomorphism $W\to U$, given by $X\mapsto \Sigma X$,
yielding isomorphisms $W/N\to U/S$ and $N/d W\to S/M$. We have thus proved the following result.

\begin{theorem}\label{quot}
\label{W} Let $A\in M_{m,n}(R)$, with rank $0<r<n$ and nullspace $S$ in $R^n$.
Let $Z$ be a reduced matrix associated to $A$, as in (\ref{ready}), say via $LA\Sigma=Z$.
Let $W$ be the free $R$-module of rank $n-r$ corresponding to $Z$ as defined in \S\ref{wpq},
and set $U=\Sigma W$ and $M=dU$.
Then $M\subseteq S\subseteq U$, where $U/S\cong W/N$ and $S/M\cong N/dW$ are as described in (\ref{form}) and (\ref{formula}), respectively.
\end{theorem}

\begin{cor}\label{nada}  We have $U=S$ if and only if $d$ divides every entry of $K$, and $S=M$ if and only if $\gcd(d,q_i)=1$, $1\leq i\leq t$,
and either $d$ is a unit or $K$ has rank $f$.
\end{cor}

\begin{proof} This follows immediately from (\ref{form}) and (\ref{formula}).
\end{proof}

\section{An improved procedure to construct a basis of $L$}\label{algo}

Here we go back to the general case and suppose that $L$ is an arbitrary lattice of rank $f$ with a proper sublattice $M$.
We assume that the list of invariant factors or elementary divisors of $L/M$ is known, and we wish to use one list or the other
to improve the process indicated in the Introduction to obtain a basis of $L$ from a given basis $\{u_1,\dots,u_f\}$ of $M$.

Let $g_1,\dots,g_s\in R$ be the unique elements, up to multiplication by units,
such that $g_1$ is not a unit, $g_s$ is not zero,
$g_1|\cdots|g_s$, and
\begin{equation}\label{sm}
L/M\cong R/Rg_1\oplus\cdots\oplus R/Rg_s.
\end{equation}
Here $i(M,L)=g_1\cdots g_s$, and we will use all of $g_1,\dots,g_s$ instead of $i(M,L)$
to obtain a basis of~$L$. The idea is to advance one invariant factor of $L/M$ at a time,
rather than one prime factor of $i(M,L)$ at a time.

According to (\ref{sm}), $S/M$ has a vector with annihilating ideal $Rg_s$. This means that there
are $a_1,\dots,a_f\in R$ such that the following extension of  (\ref{p}) holds:
\begin{equation}\label{vg}
v=\frac{a_1u_1+\cdots+a_f u_f}{g_s}\in L\text{ but }hv\notin M\text{ for any proper factor }h\text{ of }g_s.
\end{equation}
In particular, $Rv\cap M=Rg_sv$, and we set $P=Rv+M$. Thus $$P/M\cong Rv/(Rv\cap M)\cong R/Rg_s$$ is a submodule of~$L/M$.
On the other hand, it is well-known \cite[Lemma 6.8 and Theorem 6.7]{H} that any cyclic submodule of $L/M$ with annihilating ideal $Rg_s$ is complemented in $L/M$. The uniqueness of the invariant factors of $L/M$ implies that
$$
S/P\cong R/Rg_1\oplus\cdots\oplus R/Rg_{s-1}.
$$
Thus, if we can provide a way to produce a basis of $P$ from a  basis of $M$, then successively applying the above procedure
with $g_s,g_{s-1},\dots,g_1$ will yield a basis of $L$. We next indicate two ways to construct a basis of $P$ from
$\{u_1,\dots,u_f\}$ and $v$. Set $v=u_{f+1}$ and $a_{f+1}=-g_s$. Then from the first condition in (\ref{vg}), we have
\begin{equation}\label{rel}
a_1u_1+\cdots+a_f u_f+a_{f+1}u_{f+1}=0,
\end{equation}
while the second condition in (\ref{vg}) implies
\begin{equation}\label{gc}
\gcd(a_1,\dots,a_f,a_{f+1})=1.
\end{equation}

In the first way, set $a=(a_1,\dots,a_{f+1})$ and let $u$ be the column vector with vector entries $(u_1,\dots,u_{f+1})$.
Using an obvious notation, (\ref{rel}) means  $au=0$. Moreover, from (\ref{gc}) we infer the existence of $Q\in\GL_{f+1}(R)$
such that $aQ=(1,0,\dots,0)$. Setting $v=Q^{-1}u$, we have
$$
0=au=aQQ^{-1}u=(1,0,\dots,0)v.
$$
Now $v$ is a column vector, say with vector entries $(v_1,\dots,v_{f+1})$, where $v_1=0$. But $v=Q^{-1}u$ ensures
that the entries of $u$ and $v$ have the same span. Since $P$ is a lattice of rank $f$, it follows that the $f$ spanning
vectors $v_2,\dots,v_{f+1}$ must form a basis of $P$.

For the second way, we assume that $R$ is an Euclidean domain. Thus, $R$ is an integral domain endowed with a function $\delta:R\to \Z_{\geq 0}$
such that given any $a,b\in R$ with
$b\neq 0$ there are $q,r\in R$ such that $a=bq+r$, with $r=0$ or $\delta(r)<\delta(b)$. We may then use (\ref{gc}) and
the Euclidean algorithm to transform (\ref{rel}) into
\begin{equation}\label{rel2}
b_1v_1+\cdots+b_f v_f+v_{f+1}=0,
\end{equation}
where $u_1,\dots,u_{f+1}$ and $v_1,\dots,v_f$ span the same module. As above, this implies that $\{v_1,\dots,v_f\}$ is a basis of $P$.
We briefly describe the foregoing transformation. Choose $1\leq i\leq f+1$ such that $a_i\neq 0$ with
$\delta(a_i)$ is as small as possible. For notational convenience, let us assume that $i=f+1$. Dividing every
other $a_j$ by $a_{f+1}$, we obtain $a_j=q_ja_{f+1}+r_j$, where $r_j=0$ or $\delta(r_j)<\delta(a_j)$, $1\leq j\leq f$.
If every $r_j=0$ then (\ref{gc}) forces $a_{f+1}$ to be a unit, so dividing (\ref{rel}) by
$a_{f+1}$ we obtain (\ref{rel2}). Suppose at least one $r_j\neq 0$. We can re-write (\ref{rel}) in the form
$$
r_1u_1+\cdots+r_f u_f+a_{f+1}(q_1u_1+\cdots +q_f u_f+u_{f+1})=0,
$$
where $u_1,\dots,u_f,u_{f+1}$ and $u_1,\dots,u_f,q_1u_1+\cdots +q_f u_f+u_{f+1}$ span the same module, $r_j\neq 0$,
$\delta(r_j)<\delta(a_{f+1})$, and $\gcd(r_1,\dots,r_f,a_{f+1})=1$. Since $\delta$ takes only non-negative values,
repeating this process we must eventually arrive to a unit remainder, as required for (\ref{rel2}).

We next indicate how to use the elementary divisors of $L/M$ instead of its invariant factors to construct a basis of $L$.
There are more of the former than of the latter, but this is be balanced by the fact that each intermediate basis is more easily found.
Let $p\in R$ be a prime, $1\leq e_1\leq\cdots\leq e_k$, and suppose that $p^{e_1},\dots,p^{e_k}$ are the  $p$-elementary divisors of $L/M$.
Set $e=e_k$. Then $L/M$ has a vector with annihilating ideal $Rp^{e}$, which translates as follows. There are $a_1,\dots,a_f\in R$ such that the following extension of (\ref{p}) holds:
\begin{equation}
\label{vto}
v=\frac{a_1u_1+\cdots+a_f u_f}{p^{e}}\in L\text{ but }p^{e-1}v\notin M.
\end{equation}
By \cite[Lemma 6.8]{H}, any vector of $L/M$ with annihilating ideal $Rp^{e}$
has a complement in $L/M$. Thus, the preceding procedure applies, except that now we advance one $p$-elementary divisor of $L/M$
at a time. In this case, however, it is easier to pass from a basis to the next one. Indeed, since $p^{e-1}v\notin M$, we must have
$p\nmid a_i$ for some $i$, and the same argument given in the Introduction produces a basis of the span of $v,u_1,\dots,u_f$
from the basis $\{u_1,\dots,u_f\}$ of~$M$.

We finally indicate how to apply the above procedure when $L=S$, $M=d U$, and we take $\{u_1,\dots,u_f\}=\Sigma\{dV(1),\dots,dV(f)\}$.
The invariant factors of $S/M\cong N/d W$ are given in Theorem \ref{quot}, and we can obtain from
these corresponding the elementary divisors. Furthermore, Corollary \ref{nada} makes it clear when $S=M$. Observe that we can replace $L$ in (\ref{vg}) and (\ref{vto}) by $R^n$, for in that case $v\in F^n\cap T=S$.

\section{The case when $d$ is a prime}\label{dp}

We assume throughout this section that $d=p$ is a prime and set $\overline{R}=R/Rp$. In this case,
a sharpening of (\ref{form}) and (\ref{formula}) is available, and we can obtain a basis of $N$, and hence of $S=\Sigma N$, directly,
without having to resort to the procedure outlined in \S\ref{algo}.
It follows from (\ref{faltaba}) that all of $W/p W$, $W/N$ and $N/p W$ are $\overline{R}$-vector spaces, and hence completely determined by
their dimensions. Let
$\overline{K}$ be the reduction of $K$ modulo $p$. Then $W/p W\cong \overline{R}^{f}$; $N/p W$ isomorphic to the nullspace of
$\overline {K}$ by \S\ref{wpq}; and $W/N$ is isomorphic to the column space of $\overline{K}$
by \S\ref{main}. Thus
\begin{equation}
\label{s}
\dim W/N=\mathrm{rank}\, \overline{K},\; \dim N/p W=f-\mathrm{rank}\, \overline{K}.
\end{equation}
This formula is compatible with the isomorphism
$$
W/N\cong (W/p W)/(N/p W).
$$
Moreover, a careful examination of (\ref{s}) reveals that, as expected, it is in agreement with (\ref{form}) and~(\ref{formula}).

Next we show how to obtain a basis of $N$ directly from the basis $\{V(1),\dots,V(f)\}$ of $W$.
Let $H\in M_{r,f}(R)$ be such that $\overline{H}$
is the reduced row echelon form of $\overline{K}$. For simplicity of notation, let us assume that the leading
columns of $\overline{H}$ are columns $1,\dots,s$.

\begin{theorem} Consider the vectors
$$
z_{1}=pV(1),\dots,z_s=pV(s),
$$
and if $s<f$ also the vectors
$$
z_{s+i}=-(H_{1,s+i}V(1)+\cdots+H_{s,s+i}V(s))+V(s+i),\quad 1\leq i\leq f-s.
$$
Then $\{z_1,\dots,z_f\}$ is a basis of $N$ (if $s=0$ $\{z_1,\dots,z_f\}$ is simply $\{V(1),\dots,V(f)\}$).
\end{theorem}

\begin{proof} Given $\alpha=(\alpha_1,\dots,\alpha_f)\in R^f$, we have
$$
\overline{K}\overline{\alpha}=0\Leftrightarrow \overline{H}\overline{\alpha}=0
$$
and therefore (\ref{cong}) gives
$$
\alpha_1 V(1)+\cdots+\alpha_f V(f)\in N\Leftrightarrow \overline{H}\overline{\alpha}=0.
$$
Our choice of $H$ ensures that $z_1,\dots,z_f\in N$.
Let $G\in M_f(R)$ be the matrix whose columns are the coefficients of $z_1,\dots,z_f$ relative to the basis $V(1),\dots,V(f)$ of $W$.
Then $|G|=p^s$. On the other hand, $W/N\cong C(\overline{K})$ is a vector space over $\overline{R}$ of dimension $s$,
so there is a basis $\{u_1,\dots,u_f\}$ of $W$ such that $\{pu_1,\dots,pu_{s},u_{s+1},\dots,u_{f}\}$ is a basis of $N$.
It follows from \cite[Lemma 4.3]{QSS} that $\{z_1,\dots,z_f\}$ is a basis of $N$.
\end{proof}


\section{Examples}\label{exa}

(1) Consider the case $R=\Z$, $n=4$, and
$$
A=\left(\begin{array}{cccc}
2 & 3 & 5 & 4\\
3 & -5 & 2 & -7\end{array}\right).
$$
Let $B$ (resp. $C$) be the $2\times 2$ submatrix formed by the first (resp. last) two  columns of $A$ and let $D$ be
the adjoint of $B$. Then $|B|=-19$, which implies $|D|=-19$ and $|DC|\equiv |D||C|\equiv 0\mod 19$.
Multiplying $A$ on the left by $C$, we obtain the the following reduced matrix $Z$ associated to $A$:
$$
Z=\left(\begin{array}{cc} d I_2 & K\end{array}\right)=\left(\begin{array}{cccc}
-19 & 0 & -31 & 1\\
0 & -19 & -	11 & -26\end{array}\right).
$$
The reduction of $K$ modulo 19 has rank $s=1$,
since $|K|\equiv 0\mod 19$ and not all entries of $K$ are divisible by 19.
In this case, the formulas from \S\ref{dp} give
$S/19 W\cong \Z/19\Z$ and $W/S\cong \Z/19\Z$. We can use this information to obtain a
$\Z$-basis of $S$.
Indeed, by \S\ref{wpq} the vectors
$$
V(1)=(-31/19,-11/19,1,0),\; V(2)=(1/19,-26/19,0,1)
$$
form a $\Q$-basis of $T$. Moreover,  it is clear that if $\alpha_1,\alpha_2\in\Q$, then $\alpha_1 V(2)+\alpha_2 V(2)\in S$
if and only if $\alpha_1,\alpha_2\in\Z$ and
$$
-31\alpha_1+\alpha_2\equiv 0\mod 19,\; 11\alpha_1+26\alpha_2\equiv 0\mod 19.
$$
The second equation is redundant since $|K|\equiv 0\mod 19$, and the first equation is equivalent to
$$
\alpha_2\equiv 12\alpha_1\mod 19.
$$
This yields the following vectors from $S$:
$$
z_1=19 V(2)=(1,-26,0,19),\; z_2=V(1)+12V(2)=(-1,-17,1,12).
$$
The $2\times 2$ matrix formed by coordinates of $z_1,z_2$ relative to $V(1),V(2)$ is
$$
\left(\begin{array}{ccc}
0 & 1\\
19 & 12
\end{array}\right).
$$
This implies $W/(Rz_1\oplus Rz_2)\cong \Z/19\Z$, whence $\{z_1,z_2\}$ is a basis of $S$.

(2) Consider the case $R=\Z$, $r=3$, $n=6$, and
$$
A=\left(\begin{array}{cccccc}
1 & 1 & 1 & 1 & 2 & 3\\
1 & 3 & 7 & 4 & 5 & 6\\
1 & 9 & 49 & 7 & 8 & 9
\end{array}\right).
$$
Let $B$ be the $3\times 3$ submatrix formed by the first three columns of $A$. Then $B$
is a Vandermonde matrix with determinant $48$. Let $C$ be
the adjoint of~$B$. Then
$$
C=\left(\begin{array}{ccc}
84 & -40 & 4 \\
-42 & 48 & -6 \\
6 & -8 & 2
\end{array}\right).
$$
Multiplying $A$ on the left by $C$, we obtain the matrix
$$
\left(\begin{array}{cccccc}
48 & 0 & 0 & -48 & 0 & 48\\
0 & 48 & 0 & 108 & 108 & 108\\
0 & 0 & 48 & -12 & -12 & -12
\end{array}\right).
$$
Dividing every entry by 12, we obtain the following reduced matrix $Z$ associated to $A$:
$$
Z=\left(\begin{array}{cc} d I_3 & K\end{array}\right)=\left(\begin{array}{cccccc}
4 & 0 & 0 & -4 & 0 & 4\\
0 & 4 & 0 & 9 & 9 & 9\\
0 & 0 & 4 & -1 & -1 & -1
\end{array}\right).
$$
Thus (\ref{vi}) produces a free submodule $M$ of $S$ of rank 3 with basis
$$
u_1=(4,-9,1,4,0,0), u_2=(0,-9,1,0,4,0), u_3=(-4,-9,1,0,0,4).
$$
The Smith Normal Form of $K$ is $\mathrm{diag}(1,4,0)$. Here $d=4$, $f=3$ and $t=2$,
so according to (\ref{formula}), we have
$$
S/M\cong \Z/4\Z\oplus \Z/4\Z.
$$
We look for $a,b,c\in\Z$ such that
$$
v=\frac{au_1+bu_2+cu_3}{4}\in \Z^3\text{ but }2v\notin M.
$$
This translates into
$$
a+b+c\equiv 0\mod 4\text{ and }(a,b,c)\notin (2\Z)^3.
$$
Taking $(a,b,c)=(1,-1,1)$ we find the following vectors from $S$:
$$
z_1=(u_1-u_2)/4,\; z_2=(u_2-u_3)/4,\; u_3.
$$
We clearly have
$$
(\Z z_1\oplus \Z z_2\oplus \Z z_3)/M\cong \Z/4\Z\oplus \Z/4\Z,
$$
which implies that $\{z_1,z_2,z_3\}$ is a basis of $S$.

(3) Consider the case $R=\Z$, $r=3$, $n=6$, and
$$
A=\left(\begin{array}{cccccc}
12 & 24 & 36 & -4 & 12 & 44\\
24 & 36 & 12 & -2 & 10 & 20\\
36 & 12 & 24 & 0 & 20 & 44
\end{array}\right).
$$
Multiplying $A$ on the left by a suitable matrix from $\GL_3(\Q)$ yields the the following
reduced matrix associated to $A$:
$$
Z=\left(\begin{array}{cc} d I_3 & K\end{array}\right)=\left(\begin{array}{cccccc}
12 & 0 & 0 & 1 & 5 & 6\\
0 & 12 & 0 & -1 & -1 & -2\\
0 & 0 & 12 & -1 & 3 & 14
\end{array}\right).
$$
Following (\ref{vi}), we obtain a free submodule $M$ of $S$ of rank 3 having basis
$$
u_1=(-1,1,1,12,0,0), u_2=(-5,1,-3,0,12,0), u_3=(-6,2,-14,0,0,12).
$$
The Smith Normal Form of $K$ is $\mathrm{diag}(1,4,12)$. We have $d=12$, $f=3$ and $t=3$,
so (\ref{formula}) yields
$$
S/M\cong \Z/4\Z\oplus \Z/12\Z.
$$
We use (\ref{vg}) to obtain the vector
$$
v=\frac{-u_1-u_2+u_3}{12}=(0,0,-1,-1,-1,1)\in S.
$$
Then  $\{v,u_2,u_3\}$ is a basis of a module $P$ containing $M$ such that $S/P\cong\Z/4\Z$.
Applying (\ref{vg}) once again yields the vector
$$
w=\frac{4v+2u_2+u_3}{4}=(-4,1,-6,-1,5,4)\in S
$$
and the basis $\{w,v,u_2\}$ of $S$.


\end{document}